\documentclass[12pt]{amsart}

\usepackage{amssymb,latexsym}

\usepackage{graphicx}

\usepackage{enumerate}

\makeatletter

\@namedef{subjclassname@2010}{%

\textup{2010} Mathematics Subject Classification} 

\makeatother

\newtheorem{thm}{Theorem}[section]

\newtheorem{lem}[thm]{Lemma}

\theoremstyle{definition}

\newtheorem{pro}[thm]{Proposition}

\numberwithin{equation}{section}

\newcommand{\bd}{{\rm bd}}

\newcommand{\diam}{{\rm diam}}
\theoremstyle{remark}

\begin{document}

\baselineskip=17pt

\author{Marek LASSAK}
\author{Micha\l \ MUSIELAK}

\title[DIAMETER OF REDUCED SPHERICAL CONVEX BODIES]{DIAMETER OF REDUCED SPHERICAL CONVEX BODIES}

\address{Marek Lassak \newline Instytut Matematyki i Fizyki \newline University of Science and Technology\newline
85-796 Bydgoszcz, Poland}
\email{marek.lassak@utp.edu.pl}
\address{Micha\l Musielak \newline Instytut Matematyki i Fizyki \newline University of Science and Technology\newline
85-796 Bydgoszcz, Poland}
\email{michal.musielak@utp.edu.pl}

\date{}


\maketitle

\begin{abstract}The intersection $L$ of two different non-opposite hemispheres of the unit sphere $S^2$ is called a lune.  
By $\Delta (L)$ we denote the distance of the centers of the semicircles bounding $L$. 
By the thickness $\Delta (C)$ of a convex body $C \subset S^2$ we mean the minimal value of $\Delta (L)$ over all lunes $L \supset C$. 
We call a convex body $R\subset S^2$ reduced provided $\Delta (Z) < \Delta (R)$ for every convex body $Z$ being a proper subset of $R$. 
Our aim is to estimate the diameter of $R$, where $\Delta (R) < \frac{\pi}{2}$, in terms of its thickness.\end{abstract}

\section{Introduction}

\noindent
Let $S^2$ be the unit sphere of the $3$-dimensional Euclidean space $E^3$. 
A {\it great circle} of $S^2$ is the intersection of $S^2$ with any two-dimensional linear subspace of $E^{3}$. 
By a pair of {\it antipodes} of $S^2$ we mean any pair of points being the intersection of $S^2$ with a one-dimensional subspace of $E^3$.
Observe that if two different points $a, b \in S^2$ are not antipodes, there is exactly one great circle passing through them.
By the {\it arc} $ab$ connecting them we understand the shorter part of the great circle through $a$ and $b$. 
By the {\it distance} $|ab|$ of $a$ and $b$ we mean the length of the arc $ab$. 
The notion of the {\it diameter} of any set $A \subset S^2$ not containing antipodes is taken with respect to this distance and denoted by ${\rm diam}(A)$.

A subset of $S^2$ is called {\it convex} if it does not contain any pair of antipodes of $S^2$ and if together with every two points it contains the arc connecting them.
A closed convex set $C \subset S^2$ with non-empty interior is called a convex body.
Its boundary is denoted by $\bd (C)$.
If in $\bd (C)$ there is no arc, we say that the body is {\it strictly convex}.  
Convexity on $S^2$ is considered in very many papers and monographs.
For instance in \cite{BS}, \cite{DGK}, \cite{FIN}, \cite{GHS}, \cite{Ha} and \cite{VB}.

The set of points of $S^2$ in the distance at most $\rho$, where $0 < \rho \leq \frac{\pi}{2}$, from a point $c \in S^2$ is called a {\it spherical disk}, or shorter {\it a disk}, of {\it radius} $\rho$ and {\it center} $c$.  
Disks of radius $\frac{\pi}{2}$ are called {\it hemispheres}.
Two hemispheres whose centers are antipodes are called {\it opposite hemispheres}.
The set of points of a great circle of $S^2$ which are at distance at most $\frac{\pi}{2}$ from a fixed point $p$ of this great circle is called a {\it semicircle}.
We call $p$ the {\it center} of this semicircle.

Let $C\subset S^2$ be a convex body and $p \in {\rm bd} (C)$. 
We say that a hemisphere $K$ {\it supports $C$ at} $p$ provided $C \subset K$ and $p$ is in the great circle bounding $K$.

Since the intersection of every family of convex sets is also convex, for every set $A \subset S^2$ contained in an open hemisphere of $S^2$ there is the smallest convex set ${\rm conv} (A)$ containing $A$. 
We call it {\it the convex hull of} $A$. 

If non-opposite hemispheres $G$ and $H$ are different, then $L = G \cap H$ is called a {\it lune}. 
The semicircles bounding $L$ and contained in $G$ and $H$, respectively, are denoted by $G/H$ and $H/G$. 
The {\it thickness} $\Delta (L)$ of $L$ is defined as the distance of the centers of $G/H$ and $H/G$.  

After \cite{L2} recall a few notions.
For any hemisphere $K$ supporting a convex body $C\subset S^2$ we find a hemisphere $K^*$ supporting $C$ such that the lune $K\cap K^*$ is of the minimum thickness (by compactness arguments at least one such a hemisphere $K^*$ exists).
The thickness of the lune $K \cap K^*$ is called {\it the width of $C$ determined by} $K$ and it is denoted by ${\rm width}_K (C)$.
By the {\it thickness} $\Delta (C)$ of a convex body $C \subset S^2$ we understand the minimum of ${\rm width}_K (C)$ over all supporting hemispheres $K$ of $C$.
We say that a convex body $R \subset S^2$ is {\it reduced} if for every convex body $Z \subset R$ different from $R$ we have $\Delta (Z) < \Delta (R)$. 
This definition is analogous to the definition of a reduced body in normed spaces (for a survey of results on reduced bodies see
\cite{LM}). 
If for all hemispheres $K$ supporting $C$ the numbers ${\rm width}_K (C)$ are equal, we say that $C$ is {\it of constant width}.
Spherical bodies of constant width are discussed in \cite {LaMu2} and applied in \cite{HN}.

Just bodies of constant width, and in particular disks, are simple examples of reduced bodies on $S^2$. 
Also each of the four parts of a 
disk dissected by two orthogonal great circles through the center of this disk is a reduced body.
It is called a {\it quarter of a disk}.
There is a wide class of reduced odd-gons on $S^2$ (see \cite{L3}).
In particular, the regular odd-gons of thickness at most $\frac{\pi}{2}$ are reduced.

\section{Two lemmas} 

\begin{lem} \label{max}
Let $L \subset S^2$ be a lune of thickness at most $\frac{\pi}{2}$ whose bounding semicircles are $Q$ and $Q'$.
For every $u, v, z$ in $Q$ such that $v \in uz$ and for every $q \in L$ we have $|qv| \leq \max \{|qu|, |qz| \}$.
\end{lem}

\begin{proof}
If $\Delta(L)= \frac{\pi}{2}$ and $q$ is the center of $Q'$, then the distance between $q$ and any point of $Q$ is the same, and thus the assertion is obvious.
Consider the opposite case when $\Delta(L) < \frac{\pi}{2}$, or $\Delta(L)= \frac{\pi}{2}$ but $q$ is not the center of $Q'$.
Clearly, the closest point $p \in Q$ to $q$ is unique.
Observe that for $x \in Q$ the distance $|qx|$ increases as the distance $|px|$ increases. 
This easily implies the assertion of our lemma. 
\end{proof}

In a standard way, an extreme point of a convex body $C \subset S^2$ is defined as a point for which the set $C \setminus \{e\}$ is convex (see \cite{L2}). 
The set of extreme points of $C$ is denoted by $E(C)$.

\begin{lem} \label{diamE}
For every convex body $C \subset S^2$ of diameter at most $\frac{\pi}{2}$ we have ${\rm diam} (E(C)) = {\rm diam}(C)$.
\end{lem}

\begin{proof}
Clearly, ${\rm diam} (E(C)) \leq {\rm diam}(C)$.

In order to show the opposite inequality $\diam (C) \leq \diam (E(C))$, thanks to $\diam (C) = \diam (\bd(C))$, it is sufficient to show that $|cd| \leq \diam (E(C))$ for every $c, d \in \bd(C)$. 
If $c, d \in E(C)$, this is trivial.
In the opposite case, at least one of these points does not belong to $E(C)$.
If, say $d \notin E(C)$, then having in mind that $C$ is a convex body and $d \in \bd(C)$ we see that there are $e,f \in E(C)$ different from $d$ such that $d \in ef$. 
From $E(C) \subset \bd (C)$, we see that $e,f \in \bd (C)$.
Since also $d \in \bd (C)$, the arc $ef$ is a subset of $\bd (C)$. 

Recall that by Theorem 3 of \cite{L2} we have ${\rm width}_K (C) \leq {\rm diam} (C)$ for every hemisphere $K$ supporting $C$. 
In particular, for the hemisphere $K$ supporting $C$ at all points of the arc $ef$. 
Thus by the assumption that ${\rm diam} (C) \leq \frac{\pi}{2}$ we obtain ${\rm width}_K (C) \leq \frac{\pi}{2}$ for our particular $K$.
Hence we may apply Lemma \ref{max} taking this $K/K^*$ in the part of $Q$ there. 
We obtain $|cd| \leq \max \{|ce|, |cf|\}$.

If $c \in E(C)$, from $e, f \in E(C)$ we conclude that $|cd| \leq \diam (E(C))$. 
If $c \notin E(C)$, from $c \in \bd(C)$ we see that there are $g, h \in E(C)$ such that $c \in gh$. 
Similarly to the preceding paragraph we show that $|ec| \leq \max \{|eg|, |eh|\}$ and $|fc| \leq \max \{|fg|, |fh|\}$. 
By these two inequalities and by the preceding paragraph we get $|cd| \leq \max \{|eg|, |eh|, |fg|, |fh|\} \leq \diam (E(C))$, which ends the proof.
\end{proof}

The assumption that ${\rm diam} (C) \leq \frac{\pi}{2}$ is substantial in Lemma \ref{diamE}, as it follows from the example of a regular triangle of any diameter greater than $\frac{\pi}{2}$. 
The weaker assumption that $\Delta(C) \leq \frac{\pi}{2}$ is not sufficient, 
which we see taking in the part of $C$ any isosceles triangle $T$ with $\Delta(T) \leq \frac{\pi}{2}$ and the arms longer than $\frac{\pi}{2}$ (so with the base shorter than $\frac{\pi}{2}$). 
The diameter of $T$ equals to the distance between the midpoint of the base and the opposite vertex of $T$.
Hence ${\rm diam} (T)$ is over the length of each of the sides.

\medskip
\section{Diameter of reduced spherical bodies} 

The following theorem is analogous to the first part of Theorem 9 from \cite{L1} and confirms the conjecture from \cite{L3}, p. 214.
By the way, the much weaker estimate $2 \arctan \left(\sqrt 2 \tan \frac{\Delta(R)}{2}\right)$ results from Theorem 2 in \cite{Mu}.

\begin{thm} \label{ineq} 
For every reduced spherical body $R \subset S^2$ with $\Delta (R) < \frac{\pi}{2}$ we have ${\rm diam}(R) \leq {\rm arc cos} (\cos^2 \Delta (R))$.  
This value is attained if and only if $R$ is the quarter of disk of radius $\Delta(R)$.
If $\Delta (R) \geq \frac{\pi}{2}$, then ${\rm diam}(R) = \Delta (R)$.
\end{thm}

\begin{proof}
Assume that $\Delta (R) < \frac{\pi}{2}$. 
In order to show the first statement, by Lemma \ref{diamE} it is sufficient to show that the distance between any two different points $e_1, e_2$ of $E(R)$ is at most ${\rm arc cos} (\cos^2 \Delta (R))$.
Since $R$ is reduced, according to the statement of Theorem 4 in \cite{L2} there exist lunes $L_j \supset R$, where $j \in \{1,2\}$,
of thickness $\Delta(R)$ with $e_j$ as the center of one of the two semicircles bounding $L_j$ (see Figure). 
Denote by $b_j$ the center of the other semicircle bounding $L_j$.

If $e_1=b_2$ or $e_2=b_1$, then $|e_1e_2| = \Delta (R)$, which by $\Delta (R) \in (0, \frac{\pi}{2})$ is at most  $\arccos(\cos^2\Delta(R))$.
Otherwise $L_1 \cap L_2$ is a non-degenerate spherical quadrangle with points $e_1$, $b_2$, $b_1$, $e_2$ in its consecutive sides. 
Therefore, since $e_1\neq e_2$, arcs $e_1b_1$ and $e_2b_2$ intersect at exactly one point. 
Denote it by $g$. 
Observe that it may happen $b_1=b_2=g$.

Let $F$ be the great circle orthogonal to the great circle containing $e_1b_1$ and passing through $e_2$.
Since $e_2 \in  L_1$, we see that $F$ intersects $e_1b_1$.
Let $f$

\begin{center}

\includegraphics[width=3.45in]{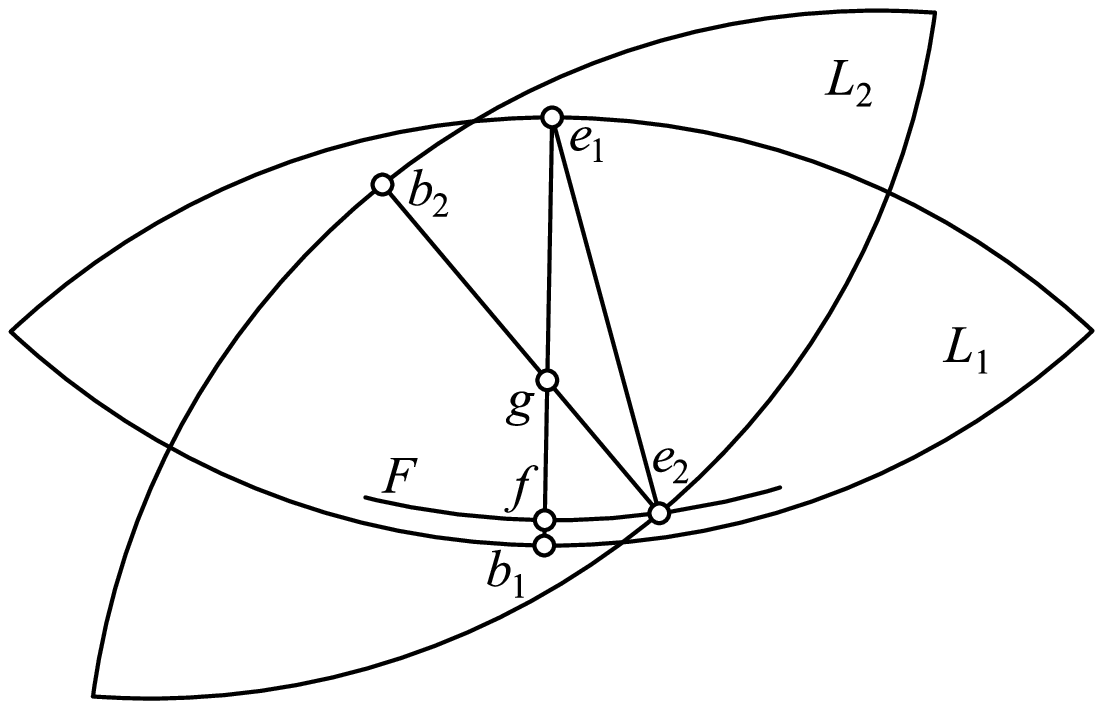} \\  

\medskip
{FIGURE. Illustration to the proof of Theorem \ref{ineq}.}

\end{center}

\noindent
be the intersection point of them.
Note that we do not exclude the case $f=b_1$.
From $|e_2b_2| = \Delta (R)$ we see that $|ge_2| \leq \Delta (R) < \frac{\pi}{2}$. 
Thus from the right spherical triangle $gfe_2$ we conclude that $|fe_2| \leq \Delta (R)$. 
Moreover, from $|e_1b_1| = \Delta (R)$ and $f \in e_1b_1$ we obtain $|fe_1| \leq \Delta (R)$.  
Consequently, from the formula $\cos k = \cos l_1 \cos l_2$ for the right spherical triangle with hypotenuse $k$ and legs $l_1, l_2$ applied to the triangle $e_1fe_2$ (again see Figure) we obtain $|e_1e_2| \leq {\rm arc cos} (\cos^2 \Delta (R))$. 
 
Observe that thanks to $\Delta (R)< \frac{\pi}{2}$, the shown inequality becomes the equality only if  $g=f=b_1=b_2$.
In this case, by Proposition 3.2 of \cite{LaMu2}, our body $R$ is a quarter of disk of radius $\Delta(R)$.
 
Finally, we show the last statement of the theorem.
Assume that $\Delta (R) \geq \frac{\pi}{2}$.
By Theorem 4.3 of \cite{LaMu2}, the body $R$ is of constant width $\Delta (R)$.

There is an arc $pq \subset R$ of length equal to ${\rm diam} (R)$.  
Clearly, $p\in \bd (R)$.
Take a lune $L$ from Theorem 5.2 of \cite{LaMu2} such that $p$ is the center of a semicircle bounding $L$.
Denote by $s$ the center of the other semicircle $S$ bounding $L$.
By the third part of Lemma 3 in
 \cite{L2}, we have $|px| < |ps|$ for every $x \in S$ different from $s$.
Hence $|px| \leq |ps|$ for every $x \in S$.
So also $|pz| \leq |ps|$ for every $z \in L$. 
In particular, $|pq|\leq |ps|$.
Since ${\rm diam} (R) = |pq|$ and $\Delta(R) = \Delta(L) = |ps|$, we obtain ${\rm diam} (R) \le \Delta(R)$.

Let us show the opposite inequality.

If ${\rm diam} (R) \leq \frac{\pi}{2}$, by Theorem 3 of \cite{L2} we get $\Delta (R) = {\rm diam} (R)$.
If ${\rm diam} (R) > \frac{\pi}{2}$, then by Proposition 1 of \cite{L2} we get $\Delta (R) \leq {\rm diam} (W)$.
Consequently, always we have $\Delta (R) \le {\rm diam} (R)$.

From the inequalities obtained in the two preceding paragraphs we get the equality ${\rm diam} (R) = \Delta (R)$.
Hence the last sentence of the theorem is proved.
\end{proof}

Proposition 3.5 of \cite{LaMu2} implies that {\it for every reduced spherical body $R$ with $\Delta (R) \leq \frac{\pi}{2}$ on $S^2$ we have ${\rm diam} (R) \leq \frac{\pi}{2}$}. 
Here is a more precise form of this statement.

\begin{pro} \label{precise}
Let $R \subset S^2$ be a reduced body.
Then ${\rm diam} (R) < {\pi \over 2}$ if and only if $\Delta (R) < {\pi \over 2}$. 
Moreover, ${\rm diam} (R) = {\pi \over 2}$ if and only if $\Delta (R) = {\pi \over 2}$. 
\end{pro}

\begin{proof}
We start with proving the first statement of our proposition.
The function $f(x) = {\rm arc \, cos} (\cos^2 x)$ is increasing in the interval $[0, \frac{\pi}{2}]$ as a composition of the decreasing functions $\arccos x$ and $\cos^2x$. 
From $f({\pi \over 2}) = {\pi \over 2}$ we conclude that in the interval $[0, {\pi \over 2})$ the function $f(x)$ accepts only the values below ${\pi \over 2}$.
Thus by Theorem \ref{ineq}, if $\Delta (R) < {\pi \over 2}$, then ${\rm diam} (R) < {\pi \over 2}$.
The opposite implication results 
from the inequality $\Delta (C) \leq {\rm diam} (C)$ for every spherical convex body $C$, which in turn follows from Theorem 3 and Proposition 1 of \cite{L2}.

Let us show the second part of our proposition.

Assume that $\diam (R) = \frac{\pi}{2}$. 
The inequality $\Delta (R) < \frac{\pi}{2}$ is impossible, by the first statement of our proposition.
Also the inequality $\Delta (R) > \frac{\pi}{2}$ is impossible, because by $\Delta (R) \leq \diam (R)$ (see Proposition 1 of \cite{L2}), it would imply $\diam (R) > \frac{\pi}{2}$ in contradiction to the assumption at the beginning of this paragraph. 
So $\Delta (R) = \frac{\pi}{2}$.

Now assume that $\Delta(R) = \frac{\pi}{2}$.
By the second statement of Theorem \ref{ineq} we get $\diam(R) = \frac{\pi}{2}$.
\end{proof}


\begin{thebibliography}{15}

\bibitem{BS}
BESAU F., SCHUSTER F., Binary operations in spherical convex geometry, \textit{Indiana Univ. Math. J.} \textbf{65} (2016), no. 4, 1263--1288.

\bibitem {DGK} 
DANZER L., GR\"UNBAUM B., KLEE V., Helly's theorem and its relatives, \textit{Proc. of Symp. in Pure Math.} vol. VII, Convexity, 1963, pp. 99--180.  

\bibitem{FIN}
 FERREIRA O. P.,  IUSEM A. N., N\'EMETH  S. Z., Projections onto convex sets on the sphere, \textit{J. Global Optim.} {\bf 57} (2013), 663--676.   
 
 \bibitem{GHS}
GAO F., HUG D., SCHNEIDER R., Intrinsic volumes and polar sets in spherical space. \textit{Math. Notae}, \textbf{41} (2003), 159–-176.

\bibitem{Ha}  
HADWIGER H., Kleine Studie zur kombinatorischen Geometrie der Sph\"are, {\it Nagoya Math. J.} {\bf 8} (1955), 45--48.

\bibitem{HN} 
HAN H., NISHIMURA T., Self-dual Wulff shapes and spherical convex bodies of constant width $\pi/2$, {\it J. Math. Soc. Japan} {\bf 69} (2017),  1475--1484. 

\bibitem{L1}   
LASSAK M., Reduced convex bodies in the plane, {\it Israel J. Math.} {\bf 70} (1990), 365--379. 

\bibitem{L2}  
LASSAK M., Width of spherical convex bodies, {\it Aequationes Math.} {\bf 89} (2015), 555--567. 

\bibitem{L3}  
LASSAK M., Reduced spherical polygons, {\it Colloq. Math.} {\bf 138} (2015), 205--216. 

\bibitem{LM}  
LASSAK M.,  MARTINI H., Reduced convex bodies in finite-dimensional normed spaces -- a survey, \textit{Results Math.} \textbf{66} (2014), No. 3-4, 405--426.

\bibitem{LaMu1}  
LASSAK M.,  MUSIELAK M., Reduced spherical convex bodies, {\it Bull. Pol. Ac. Math.} {\bf 66} (2018), 87--97. 

\bibitem{LaMu2} 
LASSAK M.,  MUSIELAK M., Spherical bodies of constant width, {\it Aequationes Math.} {\bf 92} (2018), 627--640.

\bibitem{Mu}  
MUSIELAK M., Covering a reduced spherical body by a disk, {\it Ukr. Math. J.}, to appear (see also arXiv:1806.04246).

\bibitem {VB} 
VAN BRUMMELEN G., {\it Heavenly mathematics. The forgotten art of spherical trigonometry.} Princeton University Press (Princeton, 2013).

\end{thebibliography}
\end{document}